\documentclass[11pt]{amsart}
\usepackage[utf8]{inputenc}
\usepackage{hyperref}
\usepackage{pst-all}

\newcommand{\RR}{{\mathbb R}}
\newcommand{\CC}{{\mathbb C}}
\newcommand{\NN}{{\mathbb N}}
\newcommand{\ZZ}{\mathbb Z}
\newcommand{\KK}{\mathbb K}
\newcommand{\PP}{{\mathbb P}}
\def\A{{\mathcal A}}
\def\B{{\mathcal B}}
\def\C{{\mathcal C}}
\def\G{{\mathcal G}}
\def\L{{\mathcal L}}
\def\O{{\mathcal O}}
\newcommand{\ones}{{\textbf{$\text{1}$}}} 

\DeclareMathOperator{\card}{card}

\numberwithin{equation}{section}

\theoremstyle{plain}
\newtheorem{theo}{Theorem}
\newtheorem{prop}[theo]{Proposition} 
\newtheorem{coro}[theo]{Corollary}
\newtheorem{lemma}[theo]{Lemma}

\theoremstyle{definition}

\theoremstyle{remark}
\newtheorem{remark}[theo]{Remark}

\usepackage{pgf}
\usepackage{tikz}
\usetikzlibrary{arrows,automata,positioning}

\begin{document}

\parindent = 0cm

\title[Decidability, Arithmetic Subsequences and Eigenvalues ]{Decidability, Arithmetic Subsequences and Eigenvalues of Morphic Subshifts}

\author{Fabien Durand \and Val\'erie Goyheneche}

\address{Laboratoire Ami\'enois
de Math\'ematiques Fondamentales et Appliqu\'ees, CNRS-UMR 7352,
Universit\'{e} de Picardie Jules Verne, 33 rue Saint Leu, 80000
Amiens, France.}

\email{\{fabien.durand,valerie.goyheneche\}@u-picardie.fr}

\date{\today}

\begin{abstract}
We prove decidability results on the existence of constant subsequences of uniformly recurrent morphic sequences along arithmetic progressions. 
We use spectral properties of the subshifts they generate to give a first algorithm deciding whether, given $p\in \mathbb{N}$, there exists such a constant subsequence along an arithmetic progression of common difference $p$.
In the special case of uniformly recurrent automatic sequences we explicitely describe the sets of such $p$ by means of automata. 
\end{abstract}

\maketitle


\section{Introduction}


In this paper we are concerned by arithmetic subsequences $(x_{k+np})_{n\in \mathbb{Z}}$ of morphic sequences $x = (x_n)_{n\in \mathbb{Z}}$ 
and to decision problems concerning constant such subsequences. 
Namely,

\bigskip

{\bf Input}: Two finite alphabets $\A$ and $\B$, an endomorphism $\sigma : \A^* \to  \A^*$ and a morphism $\phi : \A^* \to \B^*$.

\medskip

{\bf Question 1}: Given $p\in \mathbb{N}$ and an admissible fixed point  $x\in \A^\ZZ $ of $\sigma$, does there exist a constant subsequence $y = ((\phi (x))_{k+np})_{n\in \mathbb{Z}}$ ?

\medskip

{\bf Question 2}: Given an admissible fixed point  $x\in \A^\ZZ $, does there exist a constant subsequence $y = ((\phi (x))_{k+np})_{n\in \mathbb{Z}}$ ?

\medskip

{\bf Question 3}: Does there exist an algorithm that computes the set of integers $p\in \mathbb{N}$ satisfying the requirement of Question 1 ?

\medskip

Let us recall that Cobham showed \cite[Sec.\ 5]{Cobham:1972} arithmetic subsequences of $l$-automatic sequences are $l$-automatic. 
Durand extended this result \cite[Sec.\ I.4]{Durand:thesis} to primitive substitutive sequences.
The periodicity of such sequences is decidable and  
the length of the period can be found \cite{Harju&Linna:1986,Pansiot:1986, Honkala:1986,Durand:2013a}.
Therefore, Question 1 is decidable for $l$-automatic and primitive substitutive sequences.

We provide a different proof using dynamical systems, namely subshifts, and some of their spectral properties, for uniformly recurrent morphic sequences, and Presburger arithmetic for automatic sequences.
For uniformly recurrent automatic sequences, we show Question 2 is 
algorithmically 
decidable  and that there exists an algorithm for Question 3.
For the non uniformly recurrent automatic sequences, both questions are open. 

\medskip

In Section \ref{section:def} are given the basic definitions on morphic sequences and on subshift dynamical systems.
In Section \ref{sec:eigandsubseq} we show the relation between constant arithmetic subsequences of uniformly recurrent sequences and the spectral eigenvalues of the subshifts they generate. 
We recall some wellknown results on dynamical eigenvalues of Sturmian \cite{Kurka}, constant length substitution \cite{Dekking1978} and Toeplitz subshifts \cite{Williams}. 
As a direct consequence we observe in Section 4 that intersections of languages coming from these subshifts should be finite.
We prove the decidability of Question 1 for uniformly recurrent morphic sequences in Section \ref{section:algo}.
The decidability of Question 2 and Question 3 is proven in Section \ref{sec:constantlength} but only for uniformly recurrent automatic sequences. 
The description of the set of integers $p\in \mathbb{N}$ satisfying the requirement of Question 1 is given by means of allowed paths in a finite automaton.


\section{Definitions and background}
\label{section:def}
\subsection{Words and sequences}

In all this article, $\mathcal A$ will stand for an {\em alphabet}, that is a finite set of elements called \textit{letters}.
 A \textit{word} is an element of the free monoid $\mathcal{A}^*$ generated by $\mathcal A$. 
The neutral element of  $\mathcal{A}^*$ is called the \textit{empty word } and is denoted by $\epsilon$. 
We set $\A^+ = \A^* \setminus \{ \epsilon \}$.
For $u = u_0 u_1 u_2 u_3\cdots u_{n-1}$ in $\mathcal A^n\subset \mathcal A^*$, $n$ is called the {\em length} of $u$ and is denoted by $|u|$.
{\it One-sided sequences} are elements of 
$\mathcal A^{\mathbb N}$ and {\it sequences} are  elements of 
$\mathcal A^{\mathbb Z}$.
The sets $\mathcal A^{\mathbb N}$ and $\mathcal A^{\mathbb Z}$ are endowed with the product topology.
For convenience, we use a dot to separate negative and positive indices:
 a sequence $x\in\A^\ZZ$ would be written $\cdots x_{-2} x_{-1}.x_0 x_1 x_2\cdots$.
We set $u_{[p,q]} = u_p u_{p+1}\cdots u_q$. The integer $p$  is the \textit{occurrence} of the {\em factor} $u_{[p,q]}$ in the word $u$. If a factor has several occurrences in a word $u$, we call \textit{gap} the difference between two successive occurrences. 
We denote by $\L (u)$ the set of all factors of $u$ and we call it the {\it language of } $u$.
The number of occurrences of a word $v$ in the word $u$ is denoted by $|u|_v$.
These definitions also hold when $u$ belongs to $\mathcal A^{\ZZ}$ or $\A^\NN$.

Let $u$ and $v$ be two words in $\A^*$. 
The set $\{ x\in \A^ \ZZ : x_{[-|u|,|v|-1]}=uv\}$ is called a {\em cylinder} set and denoted $[u.v]$, or $[v]$ if $u$ is the empty word.
These sets generates the product topology of $\A^\ZZ$.

If $u = vw$ belongs to $\A^*$ or $\A^\NN$, the word $v$ is called \textit{prefix} of $u$.
An \textit{arithmetic subsequence} of $x$ in $\mathcal A^{\ZZ} $ (resp. $\A^\NN$) is a sequence of the form $(u_{k+np})_{n\in\ZZ}$ (resp.  $(u_{k+np})_{n\in\NN}$) for some $k$ and $p$. The integer $p$ is called the {\em common difference} of this arithmetic subsequence of $x$. 

The sequence $x\in\A^\NN$ (resp. in $\A^\ZZ$) is \emph{periodic} with period $p$ if there exists $p\in\NN$ such that $x_i =x_{i+p}$ for all $i\in\NN$ (resp. $\ZZ$). 
Otherwise, $u$ is \emph{non-periodic}.

A sequence $x$ is \textit{recurrent} if every factor of $x$ occurs in $x$ infinitely often.
It is \textit{uniformly recurrent} if every factor of $x$ occurs in $x$ with bounded gaps.

In the sequel $\A $ and $\B$ will stand for alphabets.

\subsection{Morphisms and matrices} 

Let $\sigma$ be a {\it morphism} 
from $\A^*$ to $\B^*$. 
When $\sigma (\A) = \B$, we say that $\sigma$ is a {\em coding}. 
If $\sigma (\A)$ is included in $\B^+$, it induces by concatenation a map
from $\A^{\mathbb{N}}$ to $\B^{\mathbb{N}}$. These two maps are also called $\sigma$.
To the morphism $\sigma$ is naturally associated the matrix $M_{\sigma} =
(m_{i,j})_{i\in \B , j \in \A }$ where $m_{i,j}$ is the number of
occurrences of $i$ in the word $\sigma(j)$.
We call it the {\em incidence matrix} of    $\sigma$.
We set $|\sigma | = \max_{a\in \A} |\sigma (a)|$ and $\langle \sigma \rangle = \min_{a\in \A} |\sigma (a)|$.

\subsection{Substitutions}


In the sequel we use the definition of substitution presented in \cite{Queffelec:2010} and the notion of substitutive sequence defined in \cite{Durand:1998}.

The \emph{language} of the endomorphism $\sigma : \A^* \to \A^*$ is the set $\mathcal{L} (\sigma ) $ of words appearing in some $\sigma^n (a)$, $a\in \A$.

If there exist a letter $a\in \A$ and a non-empty word $u\in \A^*$ such that $\sigma(a)=au$ and moreover, if $\lim_{n\to+\infty}|\sigma^n(a)|=+\infty$, then $\sigma$ is said
    to be {\em right-prolongable on $a$}.
Analogously, we define the notion to be {\it left-prolongable}.
The endomorphism $\sigma$ is {\it prolongable} on $b.a$ if it is left-prolongable on $b$, right-prolongable on $a$ and if $ba$ belongs to $\L (\sigma )$.
The endomorphism $\sigma$ is a {\em substitution} whenever it is prolongable and {\it growing} (that is, $\lim_n \langle \sigma^n \rangle = +\infty$). 
We say $\sigma$ is {\em left-proper} if there exists $a\in A$ such that $\sigma(A)$ is included in $ a A^*$.
It is {\em right-proper} if there exists $b\in A$ such that $\sigma(A)$ is included in $  A^*b$.
It is {\em proper} whenever it is both left and right proper.

A letter $c$ such that $\lim_{n\to +\infty}|\sigma^n (c)| = +\infty$ is called a {\em growing letter}.

It is an exercise to check that when $\sigma $ is right-prolongable on $a$, the sequence $(\sigma^n(aa\cdots ))_{n\ge 0}$ converges to a sequence denoted by  $\sigma^\infty(a)$. 
The map $\sigma : \A^\NN \to \A^\NN$ being continuous, $\sigma^\infty( a)$ is a fixed point of $\sigma$: $\sigma (\sigma^\infty( a)) = \sigma^\infty( a)$.
We say it is an {\em admissible one-sided fixed point}.
In the same way  $x=\sigma^\infty( b.a)$ can be defined and is a {\em two-sided} fixed point of $\sigma$.
When $ba$ belongs to $\L (\sigma )$ we say $x$ is an {\em admissible} fixed point of the endomorphism $\sigma$. 
We also say that $x$ is {\em purely morphic} (with respect to $\sigma$) or {\em purely substitutive} when $\sigma$ is a substitution. 
If $x\in \A^\mathbb{ Z}$ is purely morphic and $\phi:\A^*\to \B^*$ is a morphism then the sequence $y=\phi (x)$ is said to be a {\em morphic sequence} (with respect to $\phi$ and $\sigma$). 
We say $y$ is {\em $\alpha$-morphic} whenever the dominant eigenvalue of $\sigma $ is $\alpha$.
When $\phi $ is a coding and $\sigma $ a substitution, we say that $y$ is {\em substitutive} (with respect to $\phi $ and $\sigma$).

Whenever the matrix associated to $\sigma $ is
primitive (that is, when it has a power with positive coefficients) we say that $\sigma$ is a {\it primitive endomorphism}.  
In this situation we easily check that $\sigma^\infty ( a)$ and $\sigma^\infty ( b.a)$ are uniformly recurrent.
Moreover, one has $\mathcal{L} (\sigma ) =  \mathcal{L} (\sigma^\infty ( a)) = \mathcal{L} (\sigma^\infty ( b.a))$.

A sequence is {\em primitive substitutive} if it is substitutive with respect to a primitive substitution. 
Such sequences are uniformly recurrent.
As we will see later (Theorem \ref{thm:morphicsubstitutive}), the set of uniformly recurrent morphic sequences is exactly the set of primitive substitutive sequences.
If $|\sigma(a)| = p$ is constant for all $a\in  \A$, we say $\sigma$ is of \textit{constant-length} $p$.
In the other case, $\sigma$ is of \textit{non constant-length}.










\subsection{Dynamical systems, subshifts and eigenvalues}

For more details, we refer to Queff\'elec's book \cite{Queffelec:2010}.

A {\em topological dynamical system} is a couple $(X,T)$ where $X$ is a compact metric space and $T:X\rightarrow X$ a homeomorphism.
The topological dynamical system $(X',T')$ is a {\em factor} of $(X,T)$ whenever there exists a continuous and onto map $f : X \to X'$ such that $f\circ T = T' \circ f$.
We say $f  :(X,T) \to (X' , T)$ is a factor map. 
If $f$ is one-to-one and onto we say $(X,T)$ and $(X',T')$ are {\em isomorphic} and that $f$ is an {\em isomorphism}. 
We call \textit{orbit of} $x\in X$ the set $\O(x) =\{T^nx ; n\in\mathbb Z\}$.
We say that $(X,T)$ is \textit{minimal} if every orbit is dense in $X$. 
Equivalently, the only closed $T$-invariant sets in $X$ are $X$ and $\emptyset$.
We say it is \emph{$p$-periodic} for some $p\in \mathbb{N}$ whenever there exists $x\in X$ such that $X=\{ x,  Tx, T^2 x, \dots ,T^{p-1} x\}$ and $T^p x = x$. 
If such an integer $p$ does not exist, $(X,T)$ is said \emph{non-periodic}.
If $X$ is a Cantor set (i.e., a compact space without isolated points and having a countable base consisting of closed open sets, called {\em clopen} sets), we say $(X, T)$ is a \emph{Cantor system}.

We say that $\lambda \in \CC$ is an {\em eigenvalue} of $(X, T)$ whenever there exists a  continuous function $f : X\rightarrow \CC$ satisfying 
$f\circ T = \lambda f$. 
Such a $f$ is called an {\em eigenfunction} of $(X, T)$. 

Suppose $(X,T)$ is minimal.
Then, by compactness, $|f|$ is a constant function and $|\lambda| = 1$. 
It follows that there exists $\alpha\in\RR$ such that $\lambda=\exp(2i\pi\alpha)$. 
Such $\alpha$ are called {\it additive eigenvalues}.
They form an additive subgroup of $\mathbb{R}/\mathbb{Z}$ 

Let $\A$ be an alphabet. 
The shift map on $\A^\ZZ$ is the map $S : \A^\ZZ \to \A^\ZZ$ defined by $(Sx)_n = x_{n+1}$ for all $n\in \ZZ$.
A {\em subshift}  is a topological dynamical system $(X,S_{/X})$ where $X$ is a subset of $\A^\ZZ$.
In what follows, $S$ will stand for the shift map independently on the alphabet we consider and $S$ will be used instead of $S_{/X}$.

If $\sigma$ is a primitive substitution and $x$ one of its fixed points, we call \emph{primitive substitution subshift} generated by $\sigma$ the subshift $(\overline{\O(x)}, S)$, also denoted $(X_\sigma, S)$.
It does not depend on the choice of the fixed point $x$ \cite[Prop.\ 5.3]{Queffelec:2010}.

\section{Relation between eigenvalues of subshifts and constant arithmetic subsequences}
\label{sec:eigandsubseq}

In this section, given a minimal Cantor system, we  recall some well-known facts concerning the relations between its periodic behaviours and its eigenvalues.
This is part of the folklore of ergodic theory of such dynamical systems. 
Before we need some definitions.

Let us introduce some terminology already used to study Toeplitz subshifts \cite{Williams}.
Let $(X,T)$ be a minimal Cantor system, $x\in X$ and $U$ be a clopen subset of $X$.
We call {\em period}-$p$ {\em skeleton} of $x$ relatively to $U$ the set 
\begin{align*}
{\rm PS}_p (x , U) & = \{ k  \in \ZZ: T^{k+np} x \in  U , \; \forall n\in \mathbb{Z} \} . \\
\end{align*}
We say $p$ is an {\em essential period} of $x$ for $U$ if:
\begin{enumerate}
\item
${\rm PS}_p (x , U) \neq \emptyset$ and 
\item
$p$ divides $q$ for every $q$ satisfying ${\rm PS}_p (x, U) = {\rm PS}_p (x,U) -q$.
\end{enumerate}

The proofs of the following lemmas are left to the reader. 

\begin{lemma}
\label{lemma:rmq}
Let $x\in X$ and $U$ be a clopen subset of $X$.
Then, for all $n$, $${\rm PS}_p (T^n x, U) = {\rm PS}_p (x,U) -n.$$
\end{lemma}

\begin{lemma}
\label{lemma:empty}
For all $x,y\in X$ and all clopen set $U$ one has
$$
{\rm PS}_p (x,U ) = \emptyset  \Longleftrightarrow {\rm PS}_p (y,U ) = \emptyset .
$$ 
\end{lemma}
We call {\em periodic spectrum of} $(X,T)$ the following set
\begin{align*}
\mathbb{P} (X, T) = 
\{ p \geq 2 : p \ \hbox{\rm   is an essential period for some clopen set } U \} .
\end{align*}

The following proposition will be useful for our study. It states 
in the two first properties
that the set $\PP(X,T)$ can be interpreted as the set of denominators of additive rational eigenvalues associated to $(X,T)$.

Together with the third property, it gives a necessary condition to have constant arithmetic subsequences,
which will be detailed in the next lemma.

\begin{prop} \label{eigenvalues}
Let  $(X,T)$ be a minimal dynamical system and $p\in \NN$.
Then, the following are equivalent:
\begin{enumerate}
\item
\label{item:1}
$\lambda = \exp(2i\pi / p)$ is an eigenvalue of $(X,T)$;
\item
\label{item:2}
$p$ belongs to $\mathbb{P} (X,T)$;
\item
\label{item:3}
there exists a closed subset $V$ of $X$ such that $\{ V , T^{-1} V , \ldots, T^{-p+1}V\}$ is a partition of $X$;
\item
\label{item:4}
$X$ admits a periodic factor with essential period $p$.
\end{enumerate}
\end{prop}

\begin{remark}
If these properties are satisfied, the set of  eigenvalues associated with $(X,T)$ 
that are roots of unity
is $\{ \exp(2i k\pi /p) : p\in \PP(X,T), k\in\ZZ\}$.

It is observed, and easy to prove, in \cite[Lemma II.2]{Dekking1978} that the partition of Property \eqref{item:3} is unique up to cyclic permutation.
\end{remark}

\begin{proof}
\eqref{item:1} $\Rightarrow$ \eqref{item:2}
Suppose that $\lambda = \exp(2i\pi / p)$ is an eigenvalue of $(X,T)$ and $f:X\to \CC$ a continuous eigenfunction associated to $\lambda$. 
Let $z\in X$.
Replacing $f$ by $f/f( z)$ if needed, one can suppose $f(z) = 1$.
Let us prove that $p$ is an essential period for the clopen set $U = f^{-1} (\{ 1 \})$.
By minimality ${\rm PS}_p (x , U) $ is non empty for all $x\in X$.
Suppose ${\rm PS}_p (x, U) = {\rm PS}_p (x,U) -q$ and let $n\in {\rm PS}_p (x, U)$.
Then, $\lambda^{n-q} f(x ) = \lambda^n f(x)$. Thus $\lambda^q = 1$ and $q$ divides $p$.

\eqref{item:2} $\Rightarrow$ \eqref{item:3}
There exists some $x\in X$ and a clopen set $U$ such that ${\rm PS}_p (x , U)$ is not empty.
From Lemma \ref{lemma:empty}, the set ${\rm PS}_p (y , U)$ is not empty for all $y\in X$.

Let $V$ be the closed subset  $\{y\in X : {\rm PS}_p (y , U) = {\rm PS}_p (x , U)\}$.
It suffices to show that $V, T^{-1}V, \dots , T^{-p+1} V$ is a partition of $X$.
Observe that $T^{-p} V = V$. 
Thus $\cup_{i=0}^{p-1} T^{-i} V$ is a non-empty closed $T$-invariant set and by minimality it is $X$.
Suppose there exists $y$ belonging to $V \cap T^{-i} V$ with $0\leq i<p$.
Then,  using Lemma  \ref{lemma:rmq}, ${\rm PS}_p (y , U) = {\rm PS}_p (T^i  y , U) = {\rm PS}_p ( y , U)-i$.
Moreover, $p$ being an essential period it should divide $i$.  Contradiction.

\eqref{item:3} $\Rightarrow$ \eqref{item:4}
Let $\{V , T^{-1} V , \cdots T^{-p+1}V\}$ be a partition of $X$. 
Define the map $\pi : X\rightarrow\{0,\dots , p-1\}$ such that $\pi(x)=i$ if 
$x\in T^{-i}V$.
Then, the system  $(\mathbb{Z}/p\mathbb{Z},R)$, where $R$ is the addition of $1$ modulo $p$, is  a $p$-periodic factor of $(X,T)$.

\eqref{item:4} $\Rightarrow$ \eqref{item:1}
We can suppose the periodic factor is $(\mathbb{Z}/p\mathbb{Z},R)$ where $R$ is the addition of $1$ modulo $p$.
Choose some factor map $f$ from $(X,T)$ onto $(\mathbb{Z}/p\mathbb{Z},R)$.
Then, it suffices to consider the map $\phi : X \to \mathbb{C}$ defined by $\phi (x) = \exp(2i\pi f (x)/p)$.
\end{proof}

\begin{coro}
\label{coro:automatic}
Let $(X',T')$ be a factor of $(X,T)$.
Then $\mathbb{P} (X',T') $ is included in $\mathbb{P} (X,T) $.
\end{coro}

Let us illustrate these notions and results in the framework of minimal subshifts.
Let $(X,S)$ be a minimal subshift defined on the alphabet $\A$ and $x=(x_n)_{n\in \mathbb{Z}}\in X$. 
It is clear that $x$ has a constant arithmetic subsequence $(x_{k+np})_{n\in \mathbb{Z}}$ if and only if $k$ belongs to ${\rm PS}_p (x , [a])$ for some $a\in \A$.
We set 
\begin{align*}
{\rm Per} (x) & = \{ p\geq 1 :  \exists a \in \mathcal{A}, {\rm PS}_p (x , [a]) \not = \emptyset \} \\
& = \{p\geq 1 :  \exists a \in \mathcal{A} , \exists k  \: \forall n ,   x_{k+np} = a \} .
 \end{align*}

Let $\mathbb{P}' (X,S) $ be the set 
\[
\{ p \in \mathbb{P} (X,S) : p \ \hbox{\rm   is an essential period for some clopen set } [a], a \in \mathcal{A}  \} .
\]
and $\ZZ \PP '(X,S)$ be the set of multiples of the elements of $\PP' (X,S )$.

\begin{lemma}
\label{lemma:casessential}
Let $(X,S)$ be a minimal subshift  and $x\in X$.
Then, one has
\begin{align}
\label{obs:essper}
{\rm Per} (x) = \ZZ \PP' ( X,S )  .
\end{align}
\end{lemma}
\begin{proof}
It is clear that $\mathbb{P}' (X,S)$ is included in ${\rm Per} (x)$.
Hence $\ZZ \PP '(X,S)$ is also included in ${\rm Per} (x)$.
The converse inclusion is clear from the definitions.
\end{proof}

To answer Question 2, it is enough to answer positively Question 3. For this purpose, we should algorithmically determine the set $\PP' ( X,S ) $.
This, together with Proposition \ref{eigenvalues}, provides a strategy to find the constant arithmetic subsequences: one has to find the essential periods, and thus the eigenvalues
that are root of unity,
and then to check whether this provides such a sequence.

Let us enlight what is above considering some well-known families of minimal subshifts.

\subsection*{Toeplitz subshifts} \label{subsect:Toeplitz}
Let $x=(x_n)_{n\in \ZZ} \in \A^\ZZ$.
We say that $x$ is a \emph{Toeplitz sequence} if for all $k\in \ZZ$ there exists $p>1$ such that $x_{k+np} = x_k$ for all $n\in\ZZ$.
That is, for every $k$, we can find a constant arithmetic subsequence $(x_{k+np})_{n\in\ZZ}$.
Observe that this does not force $x$ to be periodic as $p$ is depending on $k$.

Some of the Toeplitz sequences are substitutive \cite[sec.\ 3]{CassaigneKarhumaki}. For example, consider the substitution
$\sigma$ defined by $\sigma (0) = 1000$ and $\sigma (1) =1010$.
It has a fixed point $\sigma^\omega(0.1) = 
\cdots 0010001000.10101000101010\cdots$.
Every index belongs to a constant arithmetic progression with common difference $2^p$ for some integer $p\in\NN$. 
Thus this fixed point is a Toeplitz sequence.

The following proposition makes Proposition \ref{eigenvalues} more precise in the context of Toeplitz sequences.
It asserts that eigenvalues are not only characterized by essential periods for clopen sets but for cylinder sets.

\begin{prop}\cite{Williams}
Let $x$ be a Toeplitz sequence and $(X,S)$ the subshift it generates.
The following properties are equivalent:
\begin{enumerate}
\item
$\lambda = \exp(2i\pi / p)$ is an eigenvalue of $(X,S)$;
\item
there exists an essential period $q$ for $[a]$, for some letter $a$, such that $p$ divides $q$.
\end{enumerate}
\end{prop}


\subsection*{Sturmian subshifts}

Given a sequence $x\in \A^\ZZ$, we define its complexity function as
$p_x(n) = \card\{x_{[i;i+n-1]} ; i\in\ZZ\}$, which gives the number of factors of length $n$ that occur in $x$.
The Morse-Hedlund theorem \cite{MorseHedlund1938} asserts that $x$ is non-periodic if and only if it satisfies $p_x(n)\geq n+1$ for all $n\in\NN$. 
We say that $x$ is \emph{sturmian} whenever $x$ is uniformly recurrent and $p_x(n) = n+1$ for all $n\in\NN$. 
A subshift is called \emph{sturmian} if it is of the form $(\overline{\O(x)},S)$, where $x$ is a sturmian sequence.

\begin{prop}
Let $x$ be a sturmian sequence. Then it admits no constant arithmetic subsequence.
\end{prop}

\begin{proof}
It is well-known that the set of eigenvalues of Sturmian subshifts is $\{ \exp{(2i\pi n \alpha )} : n\in \ZZ \}$ for some $\alpha \not \in \mathbb{Q}$ \cite[Section 4.5.3]{Kurka}.
We conclude using Proposition \ref{eigenvalues}.
\end{proof}

\subsection*{Substitutive subshifts}
In this section we recall some results concerning eigenvalues associated to primitive substitutions that will be exploited in the next section in order to solve Question 1 of the introduction.
A complete algebraic characterization of the eigenvalues of minimal substitution subshifts can be found in \cite{Ferenczi&Mauduit&Nogueira:1996}. 

We fix an integer $d\in\NN$ and consider the alphabet $\A = \{0,1,\ldots, d-1\}$.

We begin with two preliminary results where $\ones\in \ZZ^d$ denotes the vector $(1,1,\dots,1)$.

\begin{lemma}\label{exposantmatrice}
Let $M$ be a square $d\times d$-matrix and $p$ be a prime integer. 
The following properties are equivalent:
\begin{enumerate}
\item
\label{exposantmatrice:1}
there exists $m\in\NN$ such that $\ones M^m\in p\ZZ^d$ ;
\item
\label{exposantmatrice:2}
$\ones M^d\in p\ZZ^d$.
\end{enumerate}
\end{lemma}

\begin{proof}
We will prove that \eqref{exposantmatrice:1} $\Rightarrow$ \eqref{exposantmatrice:2}, the converse is obvious.

If $m$ is lower or equal to $d$ then the conclusion is clear.
Suppose $m$ greater than $d$.
One can suppose $m = \min\{i\in\NN : \ones M^i\in p\ZZ^d\}$.
Applying Cayley-Hamilton theorem to $M$ ensures that there exist some integers $a_0, a_1, \cdots, a_{ d-1}$ such that 
$\ones M^d = a_0 \ones + a_1 \ones M + \cdots + a_{d-1} \ones M^{d-1}$.
Then $\ones M^{m+d-1} = a_0\ones M^{m-1} + a_1\ones M^m + \cdots + a_{d-1}\ones M^{m-1+d-1}$, where every term $\ones M^i$ except $\ones M^{m-1}$ belongs to $p\ZZ^d$. It follows that 
$a_0\ones M^{m-1} $ belongs to $p\ZZ^d$. 
As $p$ is prime and $\ones M^{m-1}\not\in p\ZZ^d$, $p$ divides $a_0$.
With the same method and multiplying successively $\ones M^d$ by $M^{m-2}, M^{m-3},\ldots , M^{m-d}$, we establish that $p$ divides $a_1, a_2, \ldots , a_{d-1}$. As a consequence, $p$ divides $\ones M^d$.
\end{proof}

\begin{lemma}\label{vpmatrice}
Let $(X ,S)$ be a non-periodic subshift generated by a left-proper primitive  substitution $\sigma : \A^* \to\A^*$ with incidence matrix $M$ and let $p$ be an integer. The following properties are equivalent:
\begin{enumerate}
\item
\label{vpmatrice:1}
$\exp(2i\pi/p)$ is an eigenvalue of $(X , S)$ ;
\item
\label{vpmatrice:2}
there exists $m\in\NN$ such that $p$ divides $|\sigma^m(a)|$ for all $a\in\A$ ;
\end{enumerate}
Moreover, when $p$ is a prime number, this is equivalent to:
\begin{enumerate}
\setcounter{enumi}{2}
\item
\label{vpmatrice:3}
$\ones M^d = (|\sigma^d (a)|)_{a\in \A}$ belongs to $p\ZZ^d$, where $d = |\A|$.
\end{enumerate}
\end{lemma}

\begin{proof}
The equivalence between  \eqref{vpmatrice:1} and  \eqref{vpmatrice:2} follows from the proof of Lemma 27 in Durand's article \cite{Durand:2000}. 
The equivalence between \eqref{vpmatrice:2} and \eqref{vpmatrice:3} comes from the fact that $M^m$ is the incidence matrix of $\sigma^m$, for every $m\in\NN$. Therefore, $\ones M^m\in p\ZZ^d$ is the vector $(|\sigma^m(a)|)_{a\in \A}$. 
We conclude using Lemma \ref{exposantmatrice}.
\end{proof}

The previous lemma is stated for left-proper substitutions, but it is clear that it also holds for right-proper substitutions.

Given a minimal subshift $(X,S)$ and $p \in \mathbb P (X,S)$ two situations can occur, $\sup \{ n : p^n \in  \mathbb P (X,S) \}$ is bounded or not. 
We will need this information to describe precisely $\mathbb P (X,S)$.

To this end, we denote $\mathbb{P}\mathbb{P} (X,S) $ the set of prime numbers belonging to $\mathbb{P} (X,S)$ and
 $\PP \PP^\infty (X,S)$ the set of $p\in \PP \PP (X,S)$ such that $p^n$ belongs to $\mathbb{P} (X,S) $ for all $n$.

Lemmas \ref{vpmatrice} and \ref{matricepolycar1} respectively determine algorithmically $\mathbb{P}\mathbb{P} (X,S) $ and 
 $\PP \PP^\infty (X,S)$. 
 For $p\in \mathbb{P} \mathbb{P} (X,S) \setminus \PP \PP^\infty (X,S)$, 
Lemmas   \ref{matricepolycar2} and \ref{bornematrice} provide an algorithm to compute the integer $B(p)$ such that $p^{B(p)}$ belongs to $\mathbb{P} (X,S)$ but $p^{B(p)+1}$ does not.   
We will need this result for our decidability problem in the next section.

\begin{lemma}\label{matricepolycar1}\cite[Lemma 29]{Durand:2000}
Let $M$ be a $d\times d$ matrix and $p$ a prime number. The following properties are equivalent:
\begin{enumerate}
\item $\forall n\in \NN, \exists k\in \NN : \ones M^k \in p^n \ZZ^d$ ;
\item $p$ divides $\gcd(a_0,\dots,a_r)$, 
where 
\[r = \max\{i\in \NN : \{\ones, \ones M,\dots, \ones M^i\}\text{ is free}\}
\]
and $Q(X) = \sum_{i=0}^{r+1}a_i X^i\in \ZZ[X]$ is the characteristic polynomial of the restriction of $M$ to the vector subspace spanned by $\ones$, $\ones M$, $\dots$, $\ones M^r$.
\end{enumerate}
\end{lemma}

\begin{lemma}
\label{matricepolycar2}
Let $M$ be a $d\times d$ matrix and $p$ a prime number.
If there exist $n\in \NN$ and $i\in \NN$ such that each of $\ones M^i,\dots,\ones M^{i+d}$ belong to $p^n\ZZ^d$ and do not belong to $p^{n+1}\ZZ^d$, then for all $j\geq i, \ones M^j$ does not belong to $p^ {n+1}\ZZ^d$.

As a consequence, if there exist two integers $n$ and $k$ satisfying $\ones M^k \in p^n \ZZ^d$,
then, in particular, the vector $\ones M^{nd} $ belongs to $ p^n \ZZ^d$.
\end{lemma}

\begin{proof}
We only sketch the proof, the arguments used are those of Lemma \ref{exposantmatrice}.
By contradiction, let $j>i$ be the smallest integer such that  $\ones M^{j+d}$ belongs to $p^{n+1}\ZZ^d$. 
Let $Q(X) = \sum_{l=0}^{r+1}a_l X^l$ be the characteristic polynomial of $M$. 
Then, by Cayley-Hamilton theorem, every coefficient $a_l$, $0\leq l \leq r$, 
is a multiple of $p$, what can be set multiplying the equality by successively $M^{j+d-1}, M^{j+d-2} , \dots, M^{j}$. 
This leads to a contradiction.
\end{proof}

\begin{lemma}
\label{bornematrice}
Let $M$ be a $d\times d$ matrix and $p$ a prime number such that 
\begin{itemize}
\item
$\ones M^d\in p\ZZ^d$,
\item
$\exists n\in\NN : \forall k\in\NN, \ones M^k\not\in p^n\ZZ^d$.
\end{itemize}
Let $n_{\max} = \max\{n\in\NN : \exists k\in\NN, \ones M^k\in p^n\ZZ^d\}$
and $k_{\min} = \min\{k\in\NN : \ones M^k\in p^{n_{\max}}\ZZ^d\}$.
Let $K = \max\{k\in\NN : p^{k-1} \leq \max_{j} \sum_i M_{i,j} \}$. Then
\[
k_{\min } \leq p^d
\hbox{ and }  n_{\max} \leq  Kp^d .
\]
\end{lemma}

\begin{proof}
The inequality $k_{\min}\leq  d n_{\max} $ is a consequence of the previous lemma. 
We detail the proof of the second inequality.

We define a graph $\G$ whose vertices belong to $[0,p-1]^d\setminus \{ (0,\dots , 0)\}$ and
are given by the decomposition in base $p$ of the vectors $\ones M^k$ for $k\in\NN$.
More precisely, $\G$ is the unique oriented graph 
defined as follows:

\begin{itemize}
\item
the vector $\ones$ belongs to $\G$;
\item
there is an edge from $V$ to $W$  with label $i$ if 
$$VM =  p^k V_k + p^{k-1}V_{k-1} + \cdots + p V_1 + V_0 , $$ 
with $V_0,\dots, V_k$ in $[0,p-1]^d\setminus \{ (0,\dots , 0)\}$, and $W = V_i$ for some $i$.
\end{itemize}

Notice that the labels of the edges are bounded by $K$.

Let $N(k) = \max\{n\in\NN : \ones M^k\in p^n\ZZ^d\}$. 
The sequence $k\mapsto N(k)$ is non-decreasing and is eventually constant equal to $n_{\max}$.

The aim of this graph is to give a bound for $k_{\min }$, that is, the first $k$ such that $N(k) = n_{\max}$. 

Observe  that the growth of $N(k)$ is controlled by the paths of length $k$ starting from $\ones$.
Indeed, for $k\in\NN$, let $P_1, \dots , P_m$  be these paths.
To each path $P_i$, we associate $s (i)$, the sum of its labels.
Then, $N (k) = \min_{1\leq i \leq m} s(i)$.
Let $P_i$ be a path realizing this minimum. 
Notice that, as $k\mapsto N(k)$ is bounded by $n_{\max }$, any cycle in $P_i$ should be labelled by $0$.
Hence, deleting the cycles in $P_i$, whose weight is $0$ in $s(i)$, the remaining path has length at most $p^d$ and thus $s(i) \leq Kp^d$.
In other words, one has $k_{\min } \leq p^d$ and $n_{\max} \leq  Kp^d$.
\end{proof}

\subsection*{Constant-length substitutions}
Let us now consider some particular minimal substitution subshifts where the situation is easier to handle. 

The group of eigenvalues for minimal constant-length substitution subshifts has been determined by Dekking \cite{Dekking1978}. 
To this end he introduced the notion of height of a substitution.
Let $\sigma$ be a primitive substitution of constant-length and $x=(x_n)_{n\in \mathbb{Z}}$ one of its fixed points. 
The \emph{height} of $\sigma$ is 
$$
h(\sigma )=\max\{n\geq 1 : (n, |\sigma |) =1, n \text{ divides } g_0\},
$$ 
where $g_0 = \gcd\{n\geq 1, x_{n}=x_0\}$. 

\begin{theo}[\cite{Dekking1978}]
\label{theo:dekking}
Let $\sigma$ be a  primitive constant-length substitution and $(X,S)$ the subshift it generates.
Suppose $(X,S)$ is non-periodic.
Then, 
$$
\mathbb{P} (X , S ) =  \left\{ p\in \NN : p \: \text{divides } \:  h (\sigma )|\sigma |^n  \: \hbox{ for some }\: n \in  \NN \right\}.
$$
Moreover, the group of eigenvalues of $(X , S)$ is 
$$\left\{ 
\exp \left( 2i\pi q/p \right) :   q\in \mathbb{Z} , p \in  \mathbb{P} (X , S )  \right\} .
$$
\end{theo}

\begin{remark}
As a consequence, let us recall the well-known fact that subshifts arising from primitive constant-length substitutions only admit  eigenvalues that are roots of the unity. 

This does not hold for non constant-length substitution subshift. 
For example, consider the subshift arising from the Fibonacci substitution $\sigma : 0\to 01, 1\to 0$. 
The set of its eigenvalues is $\{\exp(2i\pi n\alpha ) : n\in\NN\}$ with $\alpha = (\sqrt{5}-1)/2$ \cite[Sec. 4.3.1]{Kurka}. The only eigenvalue that is a root of unity is 1.
\end{remark}

\textbf{Example.} Let us consider the substitution $\sigma$ defined by 
$0\to 0213 ; 1\to 1341 ; 2\to 4104 ; 3\to 0413 ; 4\to 2134$ \cite[sec.\ 6.1.1]{Queffelec:2010}. Its length is $|\sigma| = 4$. Following the algorithm of Dekking 
\cite[Remark\ II.9]{Dekking1978}, we can compute  $h(\sigma) = 3$.
We conclude that the periodic spectrum of the underlying dynamical system $(X,S)$ is $\mathbb{P} (X , S ) = \{3^{\delta}\times 2^n : \delta \in \{ 0,1\}, n\in\NN  \} $.

\subsection*{Automatic sequences}

Let  $\sigma : A^* \to \A^*$ be a substitution of constant-length $l$ and $\phi : \A^* \to \B^*$ be a coding.
Let $x$ be a fixed point (in $\A^\NN$ or $\A^\ZZ$) of the $\sigma$.
The sequence $y = \phi (x)$ is called \emph{$l$-automatic}. 
It is a substitutive sequence with respect to a constant-length substitution.
A huge litterature exists on this family of sequences \cite{Allouche&Shallit:2003}.


\section{An application to intersections of languages}
\label{section:intersection}

Below we say that a language is Toeplitz, Sturmian, morphic, automatic, ... whenever it is the language of a sequence of the same type.  

In a recent paper \cite{Rampersad&Shallit:2018} it is observed that $k$-automatic sequences and Sturmian sequences cannot have arbitrary large factors in common. 
It is not surprising as it is a straightforward consequence of the well-known results recalled in the previous section and basics on dynamical systems. 
In the same spirit the following proposition can be established. 

\begin{prop}
\label{prop:intlang}
Let $\L_1 $ and $\mathcal{L}_2 $ be two languages.
Then, $\mathcal{L}_1 \cap \mathcal{L}_2 $ is finite in all the following situations:
\begin{enumerate}
\item
\label{item:autosturm}
$\L_1$ is automatic and $\L_2$ is sturmian;
\item
\label{item:sturmtoep}
$\L_1$ is sturmian and $\L_2$ is Toeplitz;
\item
\label{item:morphicsturm}
$\L_1$ is morphic and $\L_2$ is sturmian associated to a non quadratic rotation number;
\end{enumerate}
\end{prop}

\begin{proof}
We give some hints. 
The details are left to the reader.
Assertions \eqref{item:autosturm} and \eqref{item:sturmtoep} can be deduced from the previous section applied to the subshifts generated by the languages. 
Assertion \eqref{item:morphicsturm} comes from the fact that morphic subshifts that are also sturmian  are $\alpha$-morphic for some quadratic number $\alpha$ \cite{Dartnell&Durand&Maass:2000}. 
\end{proof}

When $\L_1$ is $\alpha$-morphic primitive with $\alpha \not \in \mathbb{Z}$ and $\L_2$ is Toeplitz we would also like to conclude that $\mathcal{L}_1 \cap \mathcal{L}_2 $ is finite but we did not find obvious arguments. 
One can show that when the underlying substitution has determinant $\pm 1$ or has all its (matrix) eigenvalues with modulus greater or equal to $1$ then $\mathcal{L}_1 \cap \mathcal{L}_2 $ is finite but we leave the whole case as a question.


\section{Constant arithmetic subsequences and eigenvalues for minimal morphic subshifts}
\label{section:algo}


In this section we prove the following theorem.
Notice that its second statement corresponds to a positive answer to Question 1 for uniformly recurrent morphic sequences.

\begin{theo} 
\label{decidabilite}
Let $y\in \B^\ZZ$ be a uniformly recurrent morphic sequence with respect to the endomorphism $\sigma : \A^* \to  \A^*$ and the morphism $\phi : \A^* \to \B^*$.
Let $(Y,S)$ be the minimal subshift it generates.
Then the following properties hold.
\begin{enumerate}
\item
\label{decidabilite:1}
The periodic spectrum of $(Y,S)$ is algorithmically computable.
\item
\label{decidabilite:2}
Given an integer $p\in\NN$, it is algorithmically decidable whether $y$ contains a letter in arithmetic progression with common difference $p$.
\end{enumerate}
\end{theo}

Let us describe the way we proceed.
We will consider two cases.
We will begin considering purely substitutive sequences.
Our algorithms rely on properties concerning non-periodic sequences.
Thus we first consider the periodic case. 
It is easy to treat and provides all the letters appearing in $x$ in arithmetic progression.
We can then process, in the non-periodic case, with Algorithm 1 to compute the periodic spectrum. 
It is a requirement to apply Algorithm 2, which proves the second statement of the above theorem.

The general case of uniformly recurrent morphic sequences will come as a consequence.

\begin{remark}
Let $y\in \A^\ZZ$ be a uniformly recurrent sequence and $p$ be some positive integer.
If there exists $n_0$ such that $(y_{k+np})_{n>n_0}$ is constant, then, due to uniform recurrence, for any $z$ in the subshift generated by $y$, there exists $i\in [ 0,p) $ such that $(z_{i+np})_{n\in\ZZ}$ is constant.
As we are considering uniformly recurrent sequences in this section, it is enough to consider any admissible one-sided fixed point of $\sigma$ to check whether $x$ contains constant arithmetic subsequences.
\end{remark}

Below, the inputs are an  endomorphism $\sigma : \A^* \to \A^*$, a morphism $\phi : \A^* \to \B^*$, $x\in \A^\ZZ$ an admissible fixed point of $\sigma$ and $y=\phi (x)$ a uniformly recurrent sequence belonging to $\B^\ZZ$.
We recall it is decidable to check whether $y$ is uniformly recurrent or not, and uniformly recurrent morphic sequences are substitutive sequences with respect to primitive substitutions  \cite{Durand:2013b}. 
Thus, in Sections \ref{section:periodic}, \ref{subsec:algo1} and \ref{subsec:algo2}, we suppose the endomorphism $\sigma $ is primitive.

Let $(X,S)$ be the minimal subshift generated by $x$ and $(Y,S)$ be the minimal subshift generated by $y$. 
We set $M = M_\sigma$.
Algorithms  1 and 2 below answer positively to the decidability of  Question 1 for purely substitutive sequences with respect to primitive substitutions.

We treat the general case in Section \ref{subsec:algo3}.

\subsection{The periodic case}
\label{section:periodic}
In this section, we consider the particular case of a  left-proper primitive substitution $\sigma : \A^* \to \A^*$ right-prolongable on $a\in \A$.
Let $x=\sigma^\omega(a)$, $(X,S)$ be the minimal substitution subshift generated by $\sigma$ and $M$ its incidence matrix.

It is decidable to check whether $x$ is periodic\cite{Durand:2012}.
Moreover, if the answer is positive, this algorithm gives a period $q$ for the sequence $x$. 
As a consequence of Fine and Wilf Theorem \cite[Theorem 8.1.4]{Lothaire}, the essential period of $x$ is a divisor of $q$. 
Thus, to find this essential period, it suffices to consider the factor $x_{[0,q-1]}$ of $x$ and check if it has period $p$ for every divisor $p$ of $q$.
The smallest such period $p$ is the essential period of $x$.

Then the periodic subshift generated by $x$ has the following periodic spectrum:
$$ \PP(X,S) = \{p'\in\NN : p' \mbox{ divides } p\}. $$

Hence, 
for each $i\in\NN$, there exists a constant arithmetic subsequence starting at index $i$, with an essential period $p_i$ that divides $p$.
All the $p_i$ can be determined by studying the occurrences of letter $x_i$ in the factor $x_{[0,p-1]}$ of $x$.

%
%
%

For the sequel, we suppose that the subshift $(X,S)$ is non-periodic.

\subsection{Algorithm 1: determining $\PP (X,S)$ when $\sigma $ is left-proper} 
\label{subsec:algo1}

In order to determine $\PP (X,S)$ we first determine the set  $\PP\PP (X,S)$ of prime numbers belonging to $\PP (X,S)$. 
Observe that from Proposition \ref{eigenvalues} and Lemma \ref{vpmatrice} the set $\PP\PP (X,S) $ is finite.
This set is composed with two types of primes.
As in section \ref{sec:eigandsubseq}, let $\PP \PP^\infty (X,S)$ be the set of $p\in \PP \PP (X,S)$ such that $p^n$ belongs to $\PP (X,S) $ for all $n$.
For $p\in \PP \PP (X,S) \setminus \PP \PP^\infty (X,S)$ we set 
\[
N_{\rm max} (p) = \max \{ n : p^n \in \mathbb{P} (X,S) \} .
\]
Let 
\[
\mathbb{P} \mathbb{P}^\infty (X,S) =  \{ p_1 , \dots , p_k \} 
\text{  and  }
\mathbb{P} \mathbb{P} (X,S)\setminus 
\mathbb{P} \mathbb{P}^\infty (X,S) =   \{ q_1 , \dots , q_l \}.
\]

Observe that the set of additive eigenvalues of $(X,S)$ being a subgroup of $\mathbb{R} / \mathbb{Z}$, $ \mathbb{P}(X,S)$ is the set of integers
\begin{align}
\label{align:P(X,S)}
\prod_{1\leq i\leq k} p_i^{r_i} \prod_{0\leq j\leq l} q_j^{s_j} 
\end{align}
with  $r_i$ and $s_j$ in $\NN$ such that $s_j \in  [0 , N_{\max } (q_j)]$ for $1\leq i\leq k$, $0\leq j\leq l$.

%

\subsubsection*{Step 1. Determine $\PP\PP (X,S)$}
According to Lemma \ref{vpmatrice}, denoting by $\textbf{P}$ the set of prime numbers,
$$
\PP\PP (X,S)  = \{ p \in \textbf{P} : p \text{ divides} \gcd ((\ones M^d )_i , 1\leq i \leq d )\} ,
$$
which is clearly computable.

\subsubsection*{Step 2. Determine $\mathbb{P} \mathbb{P}^\infty (X,S)$}
We use the Lemmas \ref{vpmatrice} and \ref{matricepolycar1}.
\begin{itemize}
\item
Compute $r = \max\{i\in \NN : \{\ones, \ones M,\dots \ones ,  M^i\}\text{ is free}\}$.
\item
Compute $\tilde M$ the restriction of $M$ to the vector subspace spanned by $\ones, \ones M, \dots, \ones M^r$.
\item
Determine its characteristic polynomial
$Q(X) = \sum_{i=0}^{r+1}a_i X^i\in \ZZ[X]$.
\item
Compute $\gcd(a_0,\dots,a_r)$. 
Its prime divisors form the set $\mathbb{P} \mathbb{P}^\infty (X,S)$.
We set: 
$$
\mathbb{P} \mathbb{P}^\infty (X,S) = \{ p_1,p_2,\dots, p_k \} .
$$
where the $p_i$'s are pairwise distinct.
\end{itemize}

\subsubsection*{Step 3. Determine the maximal power $q^n$ for $q\in \PP \PP (X,S)\setminus \PP \PP^\infty (X,S)$}
We use Lemma \ref{matricepolycar2}.
Let $Q = \max (\PP \PP (X,S)\setminus \PP \PP^\infty (X,S))$.
\begin{itemize}
\item
Starting from $n=1$ :
\begin{itemize}
\item
Compute $g_n = \gcd   ( (\ones M^{nd})_1, \dots , (\ones M^{nd})_d )$.
\item
Determine $\tilde g_n = \max\{g\in\NN : g|g_n \text{ and } g\wedge p_m = 1  \text{ for } 1\leq m\leq j\}$. 
\end{itemize}
\item
Carry on till $\tilde g_n = \tilde g_{n+1}$. 
According to Lemma \ref{bornematrice}, it will halt in a finite time bounded by $KQ^d$, where $K$ is defined as in Lemma \ref{bornematrice}.

Let $\tilde g_n = q_1^{n_1}\cdots q_l^{n_l}$. 
Then, $\PP \PP (X,S)\setminus \PP \PP^\infty (X,S) = \{ q_1, \dots , q_l\}$ and $n_i = N_{\max } (q_i)$, $1\leq i \leq l$.
\end{itemize}

\subsubsection*{Step 4.  Output of the algorithm}
One obtains $\mathbb{P} (X,S)$ as described in \eqref{align:P(X,S)}.
With this output and Proposition \ref{eigenvalues} we are able to describe all the group of rational eigenvalues associated to the system $(X,S)$: they are the complex numbers $\exp(2i\pi q/p)$ with $q\in \ZZ$ and 
$p \in \mathbb{P} (X,S)$.

\begin{remark}
In the case of a 2-letter alphabet, Host \cite[sec.\ 2.3]{Host1986} established an algorithm that computes $\mathbb{P} (X,S)$ for non constant-length substitution subshifts.
We recall it below.

Let $\sigma$ be a non constant-length substitution. 
Let $M$ be its incidence matrix, with determinant $d$ and trace $T$. 
Then 
\[
\PP(X_\sigma, S) = \{p\in\NN : p \mbox{ divides some } w\times r^n \mbox{ for some } n\in\NN\},
\]
where:
\begin{itemize}
\item
$r = \gcd(d, T)$ ;
\item 
the prime divisors of $w$ are those of $|\sigma(0)|$ and $|\sigma(1)|$ that do not divide $r$,  with the same exponents as in $|\sigma(0)| - |\sigma(1)|$.
\end{itemize}
\end{remark}

\subsection{Algorithm 2: for a given $p$, checking if $x$ has a constant arithmetic subsequence with common difference $p$, when $\sigma$ is left-proper}\label{subsec:algo2}
This corresponds to the decidability of Question 1 for fixed point of primitive substitutions. 
Recall that we consider a non-periodic substitution subshift $(X,S)$ and that the inputs are given by Theorem \ref{decidabilite}.

Let $p$ be an positive integer.

\subsubsection*{Step 1. Determine $\mathbb{P}(X,S)$ as described in Algorithm 1.}
We start determining the greatest divisor  $\tilde{p}$ of $p$ that belongs to $\PP(X,S)$. 
Following the notations given in Algorithm 1, $\tilde p$ is of the form
$\tilde p = p_1^{r_1}\cdots p_k^{r_k}q_1^{s_1}\cdots q_l^{s_l}$
with $r_i\in \NN$, with $1\leq i\leq k$, and, $0\leq \beta_j\leq N_{\max }(q_j) $, with $1\leq j\leq l$.
We can easily observe that the sequence $x$ contains a letter in arithmetic progression with common difference $p$ if and only if it contains a letter in arithmetic progression with common difference $\tilde p$ (see Lemma \ref{lemma:casessential}).

\subsubsection*{Step 2. Check if $\tilde p$ corresponds to a letter in arithmetic progression in $x$.}
There exists, according to Lemma \ref{vpmatrice}, an integer $m_p$ such that $\tilde p$ divides $|\sigma^{m_p}(a)|$ for all $a\in\A$.
Due to Lemma \ref{matricepolycar2}, we can take
$m_p = d \max\{r_1, \dots, r_k, s_1, \dots, s_l\}$.
For every integer $i$ such that $0\leq i\leq \tilde p-1$,
check if $(n,a) \mapsto (\sigma^{m_p}(a))_{i+n\tilde p}$ is constant on $\{ (n,a) :  a\in \A,  0\leq i+n\tilde p\leq |\sigma^{m_p}(a)|-1\}$.
If such an integer $i$ exists, then $x$ has a letter in arithmetic progression with common difference $\tilde p$ (and thus, also with common difference $p$) starting at index $i$. If not, it does not exist any letter in arithmetic progression with common difference $\tilde p$ (and thus, neither of common difference $p$).

\medskip

\textbf{Example.}
Let $\sigma$ be the left-proper primitive substitution defined on the alphabet $\A = \{0,1\}$ by $0\mapsto 01$ and $1 \mapsto 0110$.
Its incidence matrix is 
\[
M = \left(
\begin{array}{c c}
1&2\\
1&2
\end{array}
\right)
.
\]

We can check that the unique fixed point $x\in \A^\ZZ$ of $\sigma$ is not periodic.
As $\sigma$ is left-proper and $|\A| = 2$, it suffices to compute 
$M^2 = 
\left(
\begin{array}{c c}
3 & 6\\
3 & 6
\end{array}
\right)$, which gives $\ones M^2 = (6 , 12)$. Since $\gcd(6, 12) = 6$, the prime integers in $\PP(X_\sigma,S)$ are $2$ and $3$.
Following Algorithm 1, we establish that $\PP(X_\sigma,S) = \{3^m , 2\times 3^m : m\in\NN\}$.

Remark that $\sigma(1)$ has no arithmetic progression with common difference $2$. 
As $|\sigma(0)|$ and $|\sigma(1)|$ are divisible by 2, it follows that $x$ does not have any letter in arithmetic progression with common period $2$.

Using the same method (Step 2 of Algorithm 2) with 
$$\sigma^2 : \left\{
\begin{array}{l}
0\to 010110\\
1\to 010110011001
\end{array} \right.
$$
we see that $x$ does not contains any letter in arithmetic progression with common difference $3$, but it admits arithmetic progressions with common difference $6$: 
$x_{6n} = 0$ and $x_{6n+1} = 1$  for all $n\in\ZZ$.

\subsection{Algorithm 3: for a given $p$, checking if $y$ has a constant arithmetic subsequence with common difference $p$, general case}\label{subsec:algo3}
We recall that the general case is that $y$ is a uniformly recurrent morphic sequence. 
Thus the inputs are an endomorphism $\sigma : \A^* \to \A^*$, a morphism $\phi : \A^* \to \B^*$, $x\in \A^\ZZ$ an admissible fixed point of $\sigma$ with $y=\phi (x)$.
Let $(X,S)$ and $(Y,S)$ be the respective subshifts these sequences generate. 

Observe that even if $y$ is uniformly recurrent, $\sigma $ is not necessarily primitive, it can even have erasing letters, and $\phi $ is not necessarily a coding.
Nevertheless one can find a primitive substitution $\sigma $, with an admissible fixed point $x'$, and a coding $\phi '$ such that $y= \phi ' (x')$.
Moreover, this can be done algorithmically.

\begin{theo}\label{thm:morphicsubstitutive}
\cite{Durand:2013b}
Uniformly recurrent morphic sequences are primitive substitutive sequences.
\end{theo}

The fact that this can be done algorithmically can be easily deduced from the proof of this theorem \cite[Section 5.1]{Durand:2013b} with the additional property that the substitution can be chosen left-proper.

Consequently, we suppose, without loss of generality, that $\sigma$ is a left-proper primitive substitution and $\phi$ is a coding. 
In addition, and again without loss of generality, we suppose $\phi : (X,S) \to (Y,S)$ is an isomorphism \cite[Prop.\ 31]{Durand&Host&Skau:1999,Durand:2000}, and thus that $\PP (Y,S) = \PP (X,S)$.

Then, to check whether, for a given $p>1$, there exists a constant arithmetic subsequence $(y_{k+np})_n$ equals to some letter $a$, it is necessary and sufficient to check whether $(x_{k+np})_n$ is a sequence on the alphabet $\phi^{-1} (\{ a \} )$.

Let us translate this into an algorithm.
Let $p\geq 1$.

\subsubsection*{Step 1.} 
Following the algorithms in  \cite{Durand:2013b} and \cite{Durand:2000}, we compute
\begin{itemize}
\item
a left-proper primitive substitution $\zeta:\B^* \to\B^*$ with fixed point $z$,
\item 
a coding $\psi :\B^*\to\A^*$ such that $\psi(z)=y$.
\end{itemize}
Let $(Z,S)$ be the subshift generated by $\zeta$.
The factor map $\psi : (Z,S) \to (Y,S)$ is an isomorphism \cite{Durand&Host&Skau:1999,Durand:2000}.

\subsubsection*{Step 2.} 
Apply Algorithm 1 to $\zeta$ to determine $\PP(Y,S) = \PP(Z,S)$.
Choose an integer $\tilde p \in \PP(Y,S)$ as in Step 1 of Algorithm 2. 

\subsubsection*{Step 3. Check if there exists a subsequence $(z_{k+n\tilde p})_n$ defined on an alphabet $\phi^{-1} (\{ a \} )$, $a\in \A$.}
We proceed as in the Step 2 of Algorithm 2 except that one has to check whether one of the maps $(n,b) \mapsto (\zeta^{m_p}(b))_{k+n\tilde p}$ has all its images in some $\phi^{-1} (\{ a \} )$, $a\in \A$.

If it is case for some $i$ and $a\in \A$, then $\tilde p$, and thus $p$, corresponds to a constant arithmetic subsequence of $y$. 
Otherwise, it is not. 

\medskip

\textbf{Example.}
Let $\sigma$ be the substitution defined on the alphabet $\A = \{0,1,2\}$ by 
$0\to 02$, $1\to 2$ and $2\to 10$.
Following the algorithm of Durand \cite[proof of Prop.\ 31]{Durand:2000}, we find that $(X_\sigma , S)$ is isomorphic to $(X_\tau , S)$ where $\tau$ is the left-proper substitution defined by
$1\to 6134242, 2\mapsto 61342426134242$, $3\mapsto 6134261356135$, 
$4\mapsto 613426135$, $5\mapsto 6134261356135$, $6\mapsto 613426135$, whose incidence matrix is
$$
M_\tau =
\begin{pmatrix}
1& 2& 3& 2& 3& 2\\
2& 4& 1& 1& 1& 1\\
1& 2& 3& 2& 3& 2\\
2& 4& 1& 1& 1& 1\\
0& 0& 2& 1& 2& 1\\
1& 2& 3& 2& 3& 2\\
\end{pmatrix}.
$$
If we compute the sum of each column 
of $M_\tau^6$ (to apply algorithm 1), 
we obtain the vector
$(930072, 1860144, 1675961, 1159797, 1675961, 1159797)$
whose entries greater common divisor is equal to $1$. Thus, the substitution $\sigma$ does not admit any non-trivial rational eigenvalue. As a consequence, no sequence of $X_\sigma $ has a constant arithmetic subsequence.

\section{The case of uniformly recurrent automatic sequences}
\label{sec:constantlength}

In Section \ref{sec:eigandsubseq} we defined automatic sequences.
We recall below that these sequences correspond exactly to definable sets in well-chosen extensions of Presburger arithmetic, Section \ref{subsec:autopresb}.
Then, we aim to show that for such framework Question 1 is decidable even in the non primitive case.
The proof is immediate once we recall classical results on Presburger arithmetic.
Then using the characterization of the eigenvalues of minimal subshifts generated by constant-length substitutions given by Dekking \cite{Dekking1978}, 
we answer positively to Question 2 and Question 3 for uniformly recurrent automatic sequences.

\subsection{Automatic sequences and Presburger arithmetic}
\label{subsec:autopresb}

The {\em Presburger arithmetic} \cite{presburger1929,presburger1991} on $\NN=\{ 0,1, \dots \}$ is the first order logical structure $\langle\NN, +\rangle$, on $\ZZ$ it is $\langle\ZZ , \geq  , +\rangle$.
That is the set of formulas without free variables composed with elements of $\KK = \NN$ or $\ZZ$, variables taking values in $\KK$, the addition, the equality,
\begin{itemize}
\item the connectives $\vee$ (or), $\wedge$ (and), $\neg$ (not), $\rightarrow$ (then), $\leftrightarrow$ (iff), and,
\item the quantifiers: $\forall$ (for all), $\exists$ (there exists),
\end{itemize}
We refer the reader to Rigo's book \cite{Rigo2014} for more details.
Observe that the order relation $\geq $ is definable in $\langle\NN, +\rangle$ noticing that $n\geq m$ if and only if there $k\in \NN$ such that $n = m+k$.
Thus, $\langle\NN, +\rangle = \langle\NN, \geq , +\rangle$.

A subset $E\subset\KK$ is {\em definable} if there exists a map $\phi $ from $\KK$ to the set of formula over $\langle\KK, \geq , + \rangle$ such that $E$ is the set of numbers $n\in \KK$ such that $\phi(n)$ is true.
It is well-known that the definable sets are the finite unions of arithmetic progressions and that this logical structure is decidable. 
This means that given a formula there exists an algorithm answering in finite time whether this sentence is true or false.

Fix $l\geq 2$ an integer. 
We define the function $V_l$ on $\KK$ as $V_l(0) = 1$ and $V_l(n) = l^i$ where $l^i$ is the largest power of $l$ dividing $n$ (e.g. $V_2(12) = 4$).
We consider the first-order logical structure  $\langle \KK, \geq , +, V_l\rangle$.
It is an extension of the Presburger arithmetic. 
We define formulas and $l$-definable sets as we did before. 
This structure is again decidable \cite[Theorem 4.1 and Corollary 6.2]{Bruyere&Hansel&Michaux&Villemaire:1994}. 
We say $x\in \A^\KK$ is $l$-definable if  for all $a\in \A$ there exist a formula $\phi_a$ defining the set $\{n\in\KK : x_n =a\}$ in the theory $\langle\KK, \geq , +, V_l\rangle$. 
We say the formulas $(\phi_a )_{a\in \A}$ define $x$.

\begin{theo} \label{equivPresburger}
Let $l\in\NN$ and $x\in \A^\KK$.
Then the following properties are equivalent:
\begin{enumerate}
\item
\label{equivPresbusger:2}
$x$ is $l-$automatic ;
\item
\label{equivPresbusger:3}
$x$ is $l-$definable.
\end{enumerate}
\end{theo}

In the theory $\langle\KK, \geq , +, V_l\rangle$, the property ``$x$ admits a letter $a$ in arithmetic progression with common difference $p \geq 1$'' is given by the sentence
$$ \exists a, \exists i : \forall k\in \KK,  \phi_a(i+p k) $$
and therefore is decidable, for any given $p\in\KK$.
Observe that multiplication by a constant, here $p$, is definable in the Presburger arithmetic, that is, once some $p\in \mathbb{N}$ is fixed, 
the set $S = \{ n : \exists k\in \mathbb{Z}, n = kp\}$ is definable where $kp$ is the abbreviation for $k+\cdots + k$ ($p$ times).
Whereas the multiplication of natural numbers is not definable in $\langle\ZZ, \geq , +, V_l\rangle$ \cite{Bes2001}.
That is, the set $\{ (x,y,n)\in \mathbb{K}^3 : xy=n \}$ is not definable in $\langle\KK, \geq , +, V_l\rangle$.
If it was, then $\langle\KK, \geq , +, V_l\rangle$ would include the Peano arithmetic which is known to be undecidable.

Now we suppose $x$ belongs to $\A^\KK$ and is $l$-automatic. 
Then, given some $p\geq 1$, the property ``$x$ admits some letter $a$ in arithmetic progression with common difference $p \geq 1$'' is given by the following formulas indexed by $a$: 
$$ 
 \exists i : \forall k\in \NN,  \phi_a(i+p k). 
$$

We proved the following theorem. 

\begin{prop}
Question  1 is decidable for automatic sequences. 
\end{prop}

Observe that Question 2 can be described by the following formula:
$$ 
\exists p, \exists i : \forall k\in \KK,  \phi_a(i+p k) \wedge \psi_a(-i+p (k+1)) .
$$
But here this formula uses the multiplication of two variables and thus it is not a formula from $\langle\KK , \geq , +, V_l\rangle$.
Hence we cannot conclude directly the decidability of this question.
It is possible that this statement could be rewritten into a formula from $\langle\KK , \geq , +, V_l\rangle$ but we did not find it.

\subsection{Heights and eigenvalues of minimal constant length substitutions subshifts}\label{section:height}
We recall some well-known results on minimal constant length substitution subshifts.

Let us consider a primitive constant-length substitution $\sigma$ defined on the alphabet $\A$. 
We set $d=\card \mathcal{A}$.
We denote $l$ its length and $h$ its height (possibly equal to 1). Let $x\in \A^\ZZ$ be one of its admissible fixed points and $(X,S)$  the minimal subshift generated by $\sigma$.
It is decidable to know whether $x$ is periodic or not \cite{Honkala:1986,Allouche&Rampersad&Shallit:2009}.
As the periodic case has been studied in Section \ref{section:periodic}, we suppose for the sequel that $(X,S)$ is non-periodic.

In that case, we recall (Theorem \ref{theo:dekking}) that the periodic spectrum is
\begin{align}
\label{align:dekkingspectrum}
\PP(X,S) = \{ \mbox{divisors of some } h\times l^m, \mbox{ with } m\in\NN\} .
\end{align}

Let us recall some well-known properties of the height.
Let 
\[
g_k = \gcd\{n\geq 1, x_{k+n}=x_{k}\} .
\]
Then one has the following property \cite[Rk. II.9 (ii)]{Dekking1978} 

\begin{align*}
\label{eq:dekking}
h =  & \max \{n\geq 1 : (n, |\sigma |) =1, n \text{ divides } g_k\} \\
= & \max \{n\geq 1 : (n, |\sigma |) =1, n \text{ divides } g_0\}  .
\end{align*}

For all $i\in\NN$ we set $\mathcal{A}_i = \mathcal{A}_i (x) =  \{ x_{i+nh} : n\in \mathbb{Z} \}$.

\begin{prop}\cite[Rk. II.9 (ii)]{Dekking1978}
\label{prop:hpart}
The height $h$ of $\sigma $ is algorithmically computable. 
Moreover $(\mathcal{A}_i )_{0\leq i < h}$ is a partition of $\mathcal{A}$ and it is algorithmically computable.
\end{prop}
Of course, $i\equiv j \mod h$ if and only if $\A_i = \A_j$. 
Observe that if $h=1$ then this partition is reduced to the whole alphabet $\A$. 
For the other extremal situation, $h= d$ implies $x$ is periodic.

\subsection{Periods in primitive automatic subshifts}\label{section:heightauto}
Let $y\in \mathcal{B}^\mathbb{Z}$ be a primitive automatic sequence, that is $y = \phi (x)$ where $\sigma$ is  a primitive constant-length substitution, $x\in \mathcal{A}^\mathbb{Z}$ one of its admissible fixed points and $\phi$ a coding.

Consider $(X,S)$ and $(Y,S)$ the minimal subshifts defined by $x$ and $y$ respectively.
The subshift $(Y,S)$ is clearly a factor of $(X,S)$ as $\phi $ defines a factor map from $(X,S)$ onto $(Y,S)$.    
Consequently the set $\mathbb{P} (Y,S) $ is included in $\mathbb{P} (X,S) $.

To answer Question 2, it is sufficient to answer positively Question 3. 
For this purpose, we should algorithmically determine the set ${\rm Per} (y)$.
We recall  it is equal to $\ZZ \PP' ( Y,S )$, Section \ref{sec:eigandsubseq}.
Thus, to answer Question 3 it is sufficient to determine $\PP' ( Y,S ) $ which is included in $\mathbb{P} (Y,S) $ and thus in  $\mathbb{P} (X,S) $.

We proceed as follows. 
We determine the alphabet consisting of the letters occurring in every possible arithmetic subsequences with common difference of the form $h\times l^m, m\in\NN$. 
Due to the regularity of the construction (inherited by the constant-length of the substitutions we are dealing with)  described below, all these alphabets can be represented as vertices of a directed edge-labelled graph $G = (V,E)$ and the sequence $y$ admits a constant arithmetic subsequence if and only if one of those alphabets equals to $\{b\}$, $b\in\B$.

\subsection{A graph to describe the sets  $\mathbb{P}' (X,S)$ and $\mathbb{P}' (Y,S)$}

We keep the assumptions and notations of Section \ref{section:height} and Section \ref{section:heightauto}, in particular the partition of $\A$ into the subsets $\A_i$ given by Proposition \ref{prop:hpart} for $\sigma$.

We define a graph, $G (\sigma )$, that will characterize $\mathbb{P}' (X,S)$ and $\mathbb{P}' (Y,S)$.

Let $G' = (V',E')$ be the directed graph where $V'$ is the family of subsets of $\mathcal{A}$ and where $(\mathcal{C} , \mathcal{D})$ is an edge of $E'$ whenever there exists some integer $i$, $0\leq i < l$, such that  
\[
\mathcal{D} = \{\sigma^{d} (b)_{i} : b\in \mathcal{C}   \} .
\]
Moreover we will consider the edges-labelling function $f : E ' \to\{0,\dots, l-1\}$ defined by $f(\C, \mathcal{D}) = i$.
Let $G (\sigma )=(V,E) $ be the subgraph of $G'$ where $V$ is the set of vertices that are reachable from some vertices $\mathcal{A}_i$, $i\in\{0,\dots,h-1\}$.
A {\em walk} of $G(\sigma)$ of {\em length} $i$ is a finite sequence of edges of the type $(\C_1 , \C_2)(\C_2 , \C_3) \dots (\C_{i-1} , \C_i)$.
The vertex $\C_1$ is called the {\em starting} vertex of this walk and $\C_i$ the {\em terminal} vertex. 
The {\em label} of this path is the finite sequence $(f(\C_1 , \C_2) , f(\C_2 , \C_3), \dots , f(\C_{i-1} , \C_i))$.
A walk of $G(\sigma)$ is called \emph{admissible} if it starts from one of the vertices $\A_i, 0\leq i\leq h-1$.

\medskip

\textbf{Example 1.}
Let us consider the substitution defined on the alphabet $\A = \{0,1,2,3\}$ by
$\sigma(0) = 013$,
$\sigma(1) = 102$,
$\sigma(2) = 231$,
$\sigma(3) = 320$.
It has height $h = 2$, which leads to the following partition of $\A$:
$\A_0 = \{0,3\}$,
$\A_1 = \{1,2\}$.

We obtain the following graph:

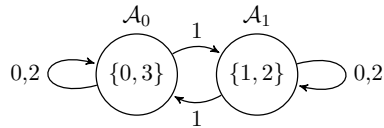
\begin{center}
\scalebox{0.8}{%
\begin{tikzpicture}[->,>=stealth',shorten >=1pt,auto,node distance=2cm,
                    semithick]
  \tikzstyle{every state}=[draw = black]

  \node[state] (A)   {$\{0,3\}$};
  \node[state] (B) [right of=A] 		 {$\{1,2\}$};
  \node[state, draw = none] (C) [above=-2mm of A] {$\A_0$};
  \node[state, draw = none] (D) [above=-2mm of B] {$\A_1$};
  
  \path (A) edge [bend left]  node {1} (B)
            edge [loop left] node {0,2} (A)
        (B) edge [loop right] node {0,2} (B)
            edge [bend left]  node {1} (A) ;
\end{tikzpicture}}
\end{center}

\textbf{Example 2.}
Now, consider the substitution $\sigma $ defined on the alphabet $\A = \{0,1,2,3\}$ by
$\sigma(0) = 01230$,
$\sigma(1) = 12301$,
$\sigma(2) = 21012$,
$\sigma(3) = 30123$.
It has height $h = 2$, which leads to the following partition of $\A$:
$\A_0 = \{0,2\}$,
$\A_1 = \{1,3\}$.
We obtain the following graph.


\medskip
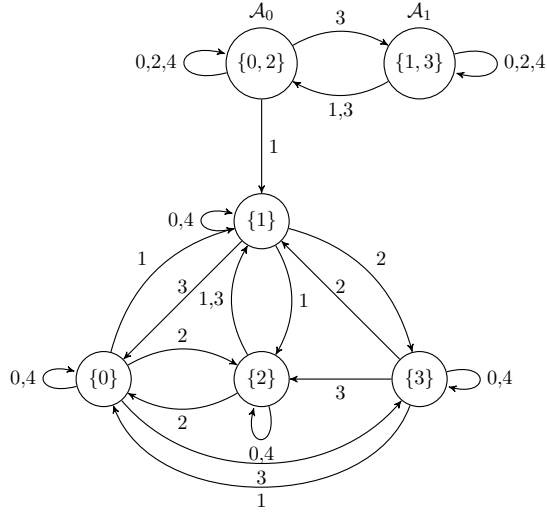
\begin{center}
\scalebox{0.7}{%
\begin{tikzpicture}[->,>=stealth',shorten >=0pt,auto,node distance=3cm, 
                    semithick]
  \tikzstyle{every state}=[draw = black]

  \node[state] (A) {$\{0,2\}$};
  \node[state] (B) [right of=A] 		 {$\{1,3\}$};
  \node[state] (C) [below of=A] {$\{1\}$} ;
  \node[state] (E) [below of=C] {$\{2\}$} ;
  \node[state] (F) [right of=E] {$\{3\}$} ;
  \node[state] (D) [left of=E] {$\{0\}$} ;
  \node[state, draw = none] (G) [above=-2mm of A] {$\A_0$};
  \node[state, draw = none] (H) [above=-2mm of B] {$\A_1$};

  \path (A) edge [bend left]  node {3} (B)
            edge [loop left] node {0,2,4} (A)
            edge [] node {1} (C)
            
        (B) edge [loop right] node {0,2,4} (B)
            edge [bend left]  node {1,3} (A) 
            
        (C) edge [loop left] node {0,4} (C)
            edge   node[above] {3} (D)
            edge  [bend left] node {1} (E)
            edge  [bend left] node {2} (F)
            
        (D) edge [loop left] node {0,4} (D)
            edge [bend left]  node {1} (C)
            edge  [bend left] node {2} (E)
            edge [bend right=50]  node[below] {3} (F)
            
        (E) edge [loop below] node {0,4} (E)
            edge [bend left]  node {1,3} (C)
            edge [bend left]  node {2} (D)
                        
        (F) edge [loop right] node {0,4} (F)
            edge  node [above]{2} (C)
            edge [bend right =-70]  node {1} (D)
            edge  node {3} (E);
\end{tikzpicture}}
\end{center}


Let us point out some properties of these graphs that will allow us to conclude with Question 3. 

\begin{lemma}
The graph $G (\sigma )$ is algorithmically computable. 
\end{lemma}
\begin{proof}
It is clear from Proposition \ref{prop:hpart} and the fact that the number of vertices is bounded by $2 ^d$.
\end{proof}

For a sequence $z$ in $\A^\KK$ we set $\A (z) = \{ z_n : n\in \KK \}$, where $\KK = \NN$ or $\ZZ$.
We call it the \emph{alphabet} of $z$.

Let $m\geq 0$ and $k$ be integers such that $0\leq k < hl^m$.
We will denote by $(k_m, k_{m-1}, \dots, k_1, k_0)$ the expansion of $k$ in base $(l^m, l^{m-1}, \dots , l, 1)$, i.e. $k = \sum_{i=0}^m k_i l^i$ with $0\leq k_m<h$ and $0\leq k_i < l$ for $0\leq i < m$.

\begin{lemma}\label{lemme:periodeschemins}
Let $m\geq 0$ and $k$ be integers such that $0\leq k < hl^m$.
Let $(k_m, k_{m-1}, \dots, k_1, k_0)$ be the expansion of $k$ in base $(l^m, l^{m-1}, \dots , l, 1)$.
Then, the set $\A ((y_{k+nhl^m})_n)$ is 
the image under $\phi$ 
of the terminal edge of the admissible walk 
in the graph $G(\sigma)$ 
starting from the vertex $\A_{k_m}$ with label $(k_{m-1}, \dots, k_1, k_0)$, read from left to right.
\end{lemma}

\begin{proof}
We proceed by recurrence on $m$.
If $m=0$, the statement is clear from the definition of the $\A_i$'s.

Suppose that the property holds for some integer $m\geq 0$.

Let $k$ be an integer such that $0\leq k < hl^{m+1}$ and $(k_{m+1}, \dots, k_1,k_0)$ be its representation in base $(l^{m+1}, l^m, \dots , l, 1)$.
Then, $k + nhl^{m+1} = k_0 + l(k_{m+1}l^{m} + k_m l^{m-1} + \cdots + k_1 +n h l^m)$.
Therefore, for all $n$, $y_{k + nhl^{m+1}} = \phi (x_{k + nhl^{m+1}})$ is the $k_0$-th letter of the word 
$\phi (\sigma(x_{k_{m+1}l^{m} + \cdots + k_1 +n h l^m}))$.
By recurrence hypothesis, the alphabet $\A ((y_{k_{m+1}l^{m} + \cdots + k_1 +n h l^m})_n)$ is the 
image under $\phi$ of the
terminal vertex of the admissible walk starting from $\A_{k_{m+1}}$ with label $(k_m,\dots , k_1)$.
By construction of the graph $G(\sigma)$, the alphabet
$\A ((y_{k + nhl^{m+1}})_n)$ 
is the image under $\phi$
of the terminal vertex of the admissible walk starting from $\A_{k_{m+1}}$ with label $(k_m,\dots k_1, k_0)$, which achieves the proof of the claim.
\end{proof}

\textbf{Example 3.}
We consider the substitution $\sigma$ defined in Example 1 above.
Let $\phi$ be the morphism defined by $0\to a, 1\to b, 2\to c$ and $3\to a$.
We apply the morphism $\phi$ to each vertex of the graph $G(\sigma)$, which could be represented by the following labelling of the graph $G(\sigma)$.

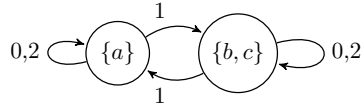
\begin{center}
\scalebox{0.8}{%
\begin{tikzpicture}[->,>=stealth',shorten >=1pt,auto,node distance=2cm,
                    semithick]
  \tikzstyle{every state}=[draw = black]

  \node[state] (A)   {$\{a\}$};
  \node[state] (B) [right of=A] 		 {$\{b,c\}$};

  \path (A) edge [bend left]  node {1} (B)
            edge [loop left] node {0,2} (A)
        (B) edge [loop right] node {0,2} (B)
            edge [bend left]  node {1} (A) ;
\end{tikzpicture}}
\end{center}

\begin{remark} \label{rem:caracpa}
As a consequence, the alphabet of any arithmetic subsequence of $y$ with common difference $hl^m$, $m\in\NN$, is 
the image under $\phi$ of
a vertex of the graph $G (\sigma)$.
In particular, there exists a constant arithmetic subsequence $(y_{k+nhl^m})_n$ equals to $a$ if and only if, 
the image under $\phi$ of
the terminal edge of the unique walk starting from vertex $\A_m$ with label $(k_{m-1},\dots, k_1,k_0)$ is $\{a\}$.
\end{remark}

\textbf{Example 1 (continued).}
From Lemma \ref{lemme:periodeschemins}, no one of its four (one-sided) fixed points admit a letter in arithmetic progression.
In fact, the substitution being primitive, it is sufficient to check this for just one of them. 
\medskip

\textbf{Example 2 (continued).}
We observe that every letter of $\A$ appears in $x$ in arithmetic progression. 
In fact, the substitution $\sigma$ being primitive and of constant length, if there exists a letter that appears in arithmetic progression, then every letter of $\A$ would occur in $x$ in arithmetic progression.
From Lemma \ref{lemme:periodeschemins} a period for the letter 1 is $2\times 5 = 10$.
Thus, the essential period should be among 2, 5 or 10.
It cannot be $2$, else either $\A_0$ or $\A_1$ would be $\{1\}$. 
Moreover, from \eqref{align:dekkingspectrum} and \eqref{obs:essper} it is a multiple of $h=2$, so the essential period for the letter $1$ is equal to $10$: 
more precisely, $x_{1+10n} = 1$ for all $n$ and this is the only constant arithmetic subsequence with common difference $10$ in $x$.

\medskip

\textbf{Example 3 (continued).} 
The letter $0$ occurs in $y$ in arithmetic progression with common difference $2$.
Notice that the letters $1$ and $2$ occur in $y$ at the same indices as in $x$, thus they do not appear in arithmetic progression.
As a consequence, the sequence $y$ has only one constant arithmetic subsequence with an essential period: $y_{2n} = 0$ for all $n\in\NN$.

\medskip

Thus from Lemma \ref{lemme:periodeschemins} and the properties of the graph 
$G (\sigma)$,
we are able to answer positively to Question 2 and Question 3 with the following proposition. 

\begin{prop}
\label{prop:grapheindices}
Let $y\in \mathcal{B}^\mathbb{Z}$ be a primitive automatic sequence, that is $y = \phi (x)$ where $\sigma$ is  a primitive constant-length substitution, $x\in \mathcal{A}^\mathbb{Z}$ one of its admissible fixed points and $\phi$ a coding.
Then, $\mathbb{Z}\mathbb{P}' (Y,S) = {\rm Per} (y) $ is the set of integers $hl^m$ such that $\phi (\A ((x_{k+nhl^m})_n))$ is a singleton for some $k$ with $0\leq k <hl^m$. 
Moreover, this set is algorithmically computable.
\end{prop}

The following proposition precise the positive answer to Question 2.

\begin{prop}
\label{prop:Q2}
The graph $G(\sigma)$ satisfies the following properties.
\begin{enumerate}
\item The sequence $y$ admits a constant arithmetic subsequence if, and only if, 
there exists a vertex $\C$ of $G(\sigma)$ whose image under $\phi$ is a singleton.
\item The sequence $y$ is periodic if, and only if, every long enough walk in the graph $G(\sigma)$ ends in a 
vertex whose image under $\phi$ is a singleton.
\end{enumerate}
\end{prop}

\begin{remark}
The quantity $b(\sigma) = \min_{\B\in V}|\B|$ is called the \emph{branching number} \cite{Kamae:1972}. 
The sequence $x$ admits a constant arithmetic subsequence if and only if $b(\sigma) = 1$.
\end{remark}

\subsection{More properties of the graph $G(\sigma )$}

We continue with the notations and assumptions of the two previous sections. 
We will now use the graph $G(\sigma)$ to characterize the set of essential periods of the letters (i.e. of $[a]$ with $a\in \A$).
We need the following proposition.

\begin{prop}\label{prop:etudegraphe}
The graph $G(\sigma)$ satisfies exactly one of the following properties.
\begin{enumerate}
\item \label{aucunepa}
It doesn't contain any singleton.
\item \label{periodique}
Every long enough walk ends in a singleton.
\item \label{panonbornees}
There exists a cycle joining vertices of cardinal $\geq 2$, with one having a singleton in his successors.
\end{enumerate}
\end{prop}

\begin{proof}
Suppose that we are neither in Case \eqref{aucunepa} nor in Case \eqref{periodique}. We will show that Case \eqref{panonbornees} is satisfied.

By hypothesis, the graph $G(\sigma)$ contains at least one singleton. Due to Remark \ref{rem:caracpa}, to this singleton corresponds a constant arithmetic subsequence of $x$, with common difference $hl^m$ for some integer $m$. 
Moreover, there exist arbitrarily long walks with vertices of cardinal greater or equal to 2. As the number of edges is finite, there exists a cycle in these vertices.

Suppose by contradiction that the vertices of this cycle do not have any singleton in their successors.
Pick $\C$ among these vertices. According to Lemma \ref{lemme:periodeschemins}, there exist two integers $k$ and $p$ such that $\C$ is the alphabet of the subsequence $(x_{k+nhl^p})_{n\in\NN}$.
Let $(\C , \C_i )$, $0\leq i < l$, be the $l$ edges starting in $\C$.
Each $\C_i$ is the alphabet of the subsequence $(x_{i+kl+nhl^{p+1}})_n$ for $0\leq i < l$. By a direct recurrence, the $j$-th successors of $\C$ contains the alphabets of each subsequence $x_{i+kl^j+nhl^{p+j}}$ for $0\leq i < l^j$. None of them are constant because $\C$ has no singleton in its successors. 
Then, the sequence $x$ contain arbitrarily long subwords with no letter in arithmetic progression. This contradicts the fact that there exist a constant arithmetic subsequence with common difference $hl^m$. 
Therefore, $\C$ has a singleton in its successors and Property \eqref{panonbornees} holds.
\end{proof}

For each of these cases, we now detail the consequences for arithmetic progressions. 
Property \eqref{aucunepa} holds if, and only if, the sequence $y$ admits a constant arithmetic subsequence (Proposition \ref{prop:Q2}). 

\begin{remark}\label{rem:graphesuccesseurs}
In the graph $G(\sigma)$, each alphabet has a cardinal greater or equal to each of its successors. In particular, every successor of a singleton is also a singleton.
\end{remark}

\begin{prop}
Every long enough walk in the graph $G(\sigma)$ ends in a singleton if, and only if, the sequence $x$ is periodic.
\end{prop}

\begin{proof}
Suppose every long enough walk in the graph $G(\sigma)$ ends in a singleton 
The graph $G(\sigma)$ contains at most $2^d$ vertices (recall that $d = \card \A$), thus every path with length $2^d$ leads to a singleton.
As a consequence, each arithmetic subsequence with common difference $h l^{2^d}$ is constant. Then $x$ is periodic and its period is a divisor of $h l^{2^d}$.

If $x$ is a periodic sequence, its period is $1$ or is a divisor of $hl^m$ for some $m$ (Lemma \ref{lemma:casessential} and Theorem \ref{theo:dekking}). 
Then each subsequence with period $h l^m$ is constant and every walk of length  greater or equal than $m$ ends in a singleton.
\end{proof}

From $G(\sigma ) $ we define a forest $F(\sigma )$, that is a finite union of (infinite) trees $T_1, \dots , T_h$.
Let $i \in \{ 0, 1 , \dots , h-1 \}$.
The root of the tree $T_i$ is $(\A_i , 0)$.
The vertices of $T_i$ are divided into {\em floors}  ($0$th floor, $1$st floor, ...).
The $0$th floor is $\{ (\A_i , 0) \}$.
The $n$th floor consists of a finite collection of elements $(\B , n)$ where $(\B' , \B)$ is an edge of $G(\sigma )$ with  $(\B' , n-1)$ belonging to the $(n-1)$th floor, and, edges are the couples $((\B' , n-1) , (\B , n))$.
Roughly speaking in $F (\sigma )$ each vertex has the same successors as in graph $G(\sigma)$. 
Admissible walks in $F (\sigma )$ are starting from some $\A_i$.
We denote $s_m$ the number of vertices $(\C , m)$ where $\C$ is a singleton. 

Let us make an observation we will use in the proof of the next proposition.
Each such vertex $(\C , m)$ corresponds to  exactly one constant arithmetic subsequences of $x$ with common difference $hl^m$ (Proposition \ref{prop:grapheindices}) and produces $l$ distinct constant arithmetic subsequences of $x$ with common difference $hl^{m+1}$.
As a consequence, one has $s_{m+1}>ls_m$ if, and only if, the sequence $x$ has an essential period, for some letter, greater than $hl^m$ and dividing $hl^{m+1}$. In the converse case, we have $s_{m+1} = l s_m$.

\begin{prop}\label{prop:bouclesingleton}
There exists in the graph $G(\sigma)$ a cycle joining vertices of cardinal $\geq 2$, with one having a singleton in his successors, if, and only if, the essential periods for letters of $x$ are unbounded.
\end{prop}

\begin{proof}
Suppose there exists in $G(\sigma)$ a cycle joining vertices of cardinal $\geq 2$ with one having a singleton in his successors.
Let $\B$ be a vertex, i.e. an alphabet, of this cycle. 
We will prove that, infinitely often, $s_{m+1}> l s_m$.
In fact only one such $m$ would be sufficient from the observation made above.

Let $N$ be the length of a shortest path joining $\B$ to a singleton $\{a\}$.
As $\B$ belongs to a cycle, it will appear infinitely often in the forest $F(\sigma)$. We denote $(b_k)_{k\in\NN}$ the levels (in increasing order) containing $\B$. 
Then, each level $b_k+N$ contains the singleton $\{a\}$ as a direct successor of an alphabet with cardinal $\geq 2$. 
The number of singletons at level $b_k+N$ is $s_{b_k+N}>s_{b_k+N-1} l$ (we have $l$ successors of the singletons of level $b_k+N-1$ and at least the singleton $\{a\}$).
Therefore, for each $k\in \NN$, $x$ has a constant arithmetic subsequence with common difference greater than $h l^{b_k+N-1}$ and dividing $hl^{b_k+N}$.

Suppose that the set of essential periods for letters is unbounded. For each essential period $p_k$ (numbered in increasing order), let $n_k $ be the smallest integer such that $p_k$ divides $hl^{n_k}$. 
Of course, the sequence  $(n_k)_k$ goes to infinity with $(p_k)_k$.
Thus, this essential period corresponds to a singleton in the level $n_k$ that is the direct descendant of some $(\B_k , k)$ with $\B_k$ having a cardinal greater or equal to $2$.
As the number of distinct alphabets in the forest is finite, there exists some $k$ such that $\B_k$ appears in this forest an infinite number of times. 
Thus, it belongs to a cycle of $G(\sigma)$ or is the descendant of such a cycle, whose elements have cardinal greater or equal than the cardinal of $\B_k$.
This ends the proof.
\end{proof}

\textbf{Example 4.}
Let us consider the substitution:
$0\mapsto 01 , 1\mapsto 20 , 2\mapsto 13 , 3\mapsto 12$.
It has length $2$ and height $3$, which lead to the alphabets:
$\A_0 = \{0\}$,
$\A_1 = \{1\}$ and
$\A_2 = \{2,3\}$.
We obtain the following graph.

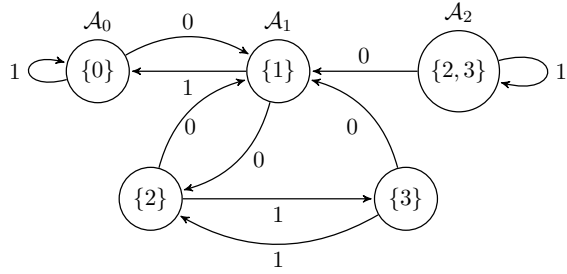
\begin{center}
\scalebox{0.8}{%
\begin{tikzpicture}[->,>=stealth',shorten >=1pt,auto,node distance=3cm,
                    semithick]
  \tikzstyle{every state}=[draw = black]

  \node[state] (B)   {$\{1\}$};
  \node[state] (A) [left of=B] 		 {$\{0\}$};
  \node[state] (C) [right of=B] 		 {$\{2,3\}$};
  \node[state] (D) [below left of=B] {$\{2\}$};
  \node[state] (E) [below right of=B]  {$\{3\}$};
  \node[state, draw = none] (F) [above=-2mm of A] {$\A_0$};
  \node[state, draw = none] (G) [above=-2mm of B] {$\A_1$};
  \node[state, draw = none] (H) [above=-2mm of C] {$\A_2$};

  \path (A) edge [bend left] node {0} (B)
            edge [loop left] node {1} (A)
        (B) edge node {1} (A)
            edge [bend left] node[ right=1mm]  {0} (D)
        (C) edge node[above] {0} (B)
            edge [loop right] node {1} (C)
        (D) edge [bend left] node[below]  {0} (B)
            edge  node [below] {1} (E)
        (E) edge [bend right] node {0} (B)
            edge [bend left] node {1} (D) ;
\end{tikzpicture}}
\end{center}

We are in Case \eqref{panonbornees} of Proposition \ref{prop:etudegraphe}, thus due to Proposition \ref{prop:bouclesingleton}, the essential periods are unbounded.
This can also be seen on the following forest.

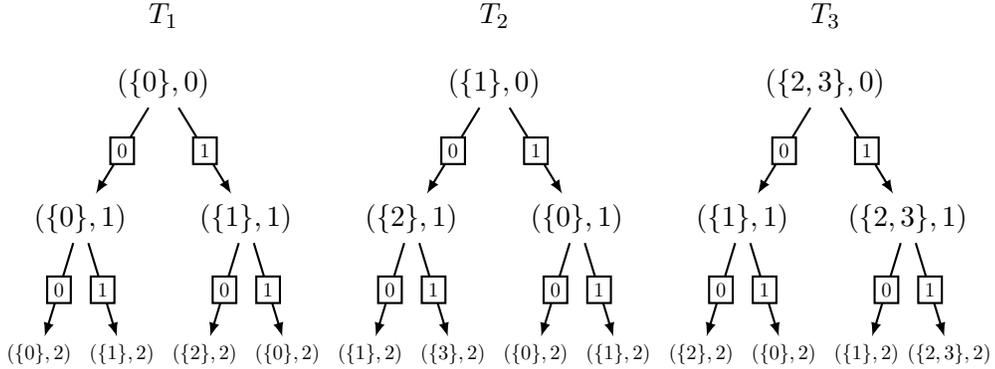
\begin{center}
\begin{tikzpicture}[xscale=0.5,yscale=0.6]
\tikzstyle{fleche}=[->,>=latex,thick]
\tikzstyle{noeud}=[rectangle]
\tikzstyle{feuille}=[rectangle,scale=0.7]
\tikzstyle{etiquette}=[midway,fill=white,draw,scale=0.7]
\def\DistanceInterNiveaux{3}
\def\DistanceInterFeuilles{2}
\def\NiveauA{(-0.5)*\DistanceInterNiveaux}
\def\NiveauB{(-1)*\DistanceInterNiveaux}
\def\NiveauC{(-2)*\DistanceInterNiveaux}
\def\NiveauD{(-3)*\DistanceInterNiveaux}
\def\InterFeuilles{(1.1)*\DistanceInterFeuilles}
\node[noeud] (T1) at ({(1.5)*\InterFeuilles},{\NiveauA}) {$T_1$};
\node[noeud] (T2) at ({(5.5)*\InterFeuilles},{\NiveauA}) {$T_2$};
\node[noeud] (T3) at ({(9.5)*\InterFeuilles},{\NiveauA}) {$T_3$};
\node[noeud] (Ra) at ({(1.5)*\InterFeuilles},{\NiveauB}) {$(\{0\},0)$};
\node[noeud] (Raa) at ({(0.5)*\InterFeuilles},{\NiveauC}) {$(\{0\},1)$};
\node[feuille] (Raaa) at ({(0)*\InterFeuilles},{\NiveauD}) {$(\{0\},2)$};
\node[feuille] (Raab) at ({(1)*\InterFeuilles},{\NiveauD}) {$(\{1\},2)$};
\node[noeud] (Rab) at ({(2.5)*\InterFeuilles},{\NiveauC}) {$(\{1\},1)$};
\node[feuille] (Raba) at ({(2)*\InterFeuilles},{\NiveauD}) {$(\{2\},2)$};
\node[feuille] (Rabb) at ({(3)*\InterFeuilles},{\NiveauD}) {$(\{0\},2)$};
\node[noeud] (Rb) at ({(5.5)*\InterFeuilles},{\NiveauB}) {$(\{1\},0)$};
\node[noeud] (Rba) at ({(4.5)*\InterFeuilles},{\NiveauC}) {$(\{2\},1)$};
\node[feuille] (Rbaa) at ({(4)*\InterFeuilles},{\NiveauD}) {$(\{1\},2)$};
\node[feuille] (Rbab) at ({(5)*\InterFeuilles},{\NiveauD}) {$(\{3\},2)$};
\node[noeud] (Rbb) at ({(6.5)*\InterFeuilles},{\NiveauC}) {$(\{0\},1)$};
\node[feuille] (Rbba) at ({(6)*\InterFeuilles},{\NiveauD}) {$(\{0\},2)$};
\node[feuille] (Rbbb) at ({(7)*\InterFeuilles},{\NiveauD}) {$(\{1\},2)$};
\node[noeud] (Rc) at ({(9.5)*\InterFeuilles},{\NiveauB}) {$(\{2,3\},0)$};
\node[noeud] (Rca) at ({(8.5)*\InterFeuilles},{\NiveauC}) {$(\{1\},1)$};
\node[feuille] (Rcaa) at ({(8)*\InterFeuilles},{\NiveauD}) {$(\{2\},2)$};
\node[feuille] (Rcab) at ({(9)*\InterFeuilles},{\NiveauD}) {$(\{0\},2)$};
\node[noeud] (Rcb) at ({(10.5)*\InterFeuilles},{\NiveauC}) {$(\{2,3\},1)$};
\node[feuille] (Rcba) at ({(10)*\InterFeuilles},{\NiveauD}) {$(\{1\},2)$};
\node[feuille] (Rcbb) at ({(11)*\InterFeuilles},{\NiveauD}) {$(\{2,3\},2)$};
\draw[fleche] (Ra)--(Raa) node[etiquette] {$0$};
\draw[fleche] (Raa)--(Raaa) node[etiquette] {$0$};
\draw[fleche] (Raa)--(Raab) node[etiquette] {$1$};
\draw[fleche] (Ra)--(Rab) node[etiquette] {$1$};
\draw[fleche] (Rab)--(Raba) node[etiquette] {$0$};
\draw[fleche] (Rab)--(Rabb) node[etiquette] {$1$};
\draw[fleche] (Rb)--(Rba) node[etiquette] {$0$};
\draw[fleche] (Rba)--(Rbaa) node[etiquette] {$0$};
\draw[fleche] (Rba)--(Rbab) node[etiquette] {$1$};
\draw[fleche] (Rb)--(Rbb) node[etiquette] {$1$};
\draw[fleche] (Rbb)--(Rbba) node[etiquette] {$0$};
\draw[fleche] (Rbb)--(Rbbb) node[etiquette] {$1$};
\draw[fleche] (Rc)--(Rca) node[etiquette] {$0$};
\draw[fleche] (Rca)--(Rcaa) node[etiquette] {$0$};
\draw[fleche] (Rca)--(Rcab) node[etiquette] {$1$};
\draw[fleche] (Rc)--(Rcb) node[etiquette] {$1$};
\draw[fleche] (Rcb)--(Rcba) node[etiquette] {$0$};
\draw[fleche] (Rcb) -- (Rcbb) node[etiquette] {$1$};
\end{tikzpicture}
\end{center}

We easily see that a new constant arithmetic progression appears at each level as a successor of vertex $(\{2,3\},2)$. 
The length of $\sigma$ is a prime number, thus the set of essential periods is exactly $\{3\times 2^n : n\in\NN\}$.

\subsection{Comments for words occurring in arithmetic progressions}
In this work we concentrated on letters occurring periodically in substitutive sequences. 
The same questions could be asked for words. 
Using the substitutions on the words of length $n$ (see \cite{Queffelec:2010}) it is clear that the main results we obtained, i.e. Theorem \ref{decidabilite} and Propositions \ref{prop:grapheindices} and \ref{prop:Q2}, can be applied to words.
This is left as an exercise. 

\section{Questions}
We leave open our three questions for the substitutive sequences (that are not uniformly recurrent). 
For example, consider the subshift $(X,S)$ generated by the primitive substitution  $0\mapsto0120$, $1\mapsto 121$, $2\mapsto  200$. 
Let $x$ be an admissible fixed point.
One has $\PP (X,S)= \{ 2^m : m\geq 0\}$ .
By computer checking we did not find constant arithmetic subsequence for $x$ with a period less than $2^{20}$ but we do not know whether it exists for greater periods.

\bigskip

{\bf Acknowledgements.}
We would like to thank Anna Frid for helpful discussions that were at the origin of the definition of the graph $G(\sigma )$ above. 
This work was partially supported by the MathAmSud project Dynamics of Cantor systems: computability, combinatorial and geometric aspects (17-MATH-01).

\bibliographystyle{alpha}
\bibliography{Biblio}

\end{document}